\documentclass[a4paper,12pt]{article}
\usepackage[utf8]{inputenc}

\title{Tameness Properties in Multiplicative Valued Difference Fields with Lift and Section}
\author{Christoph Kesting}
\date{\today}

\usepackage{amsmath}
\usepackage{eufrak}
\usepackage[T1]{fontenc}
\usepackage[utf8]{inputenc}
\usepackage[ngerman, english]{babel}
\usepackage{amsthm}
\usepackage{amssymb}
\usepackage{color} 
\usepackage{tikz-cd}
\usepackage{verbatim} 
\usetikzlibrary{babel}

\newtheorem{prop}{Proposition}[section]
\newtheorem{lem}[prop]{Lemma}
\newtheorem{cor}[prop]{Corollary}
\newtheorem{thm}[prop]{Theorem}
\newtheorem*{satz*}{Satz}
\theoremstyle{definition}
\newtheorem{defi}[prop]{Definition}

\theoremstyle{remark}
\newtheorem{rem}[prop]{Remark}
\newtheorem*{ack}{Acknowledgement}

\newtheorem{fact}[prop]{Fact}

\newcommand{\lan}{\mathcal{L}}
\newcommand{\lann}{\mathcal{L}}
\newcommand{\iso}{\cong}

\newcommand{\K}{\mathcal{K}}
\newcommand{\cM}{\mathcal{M}}
\newcommand{\cN}{\mathcal{N}}
\newcommand{\Kl}{{\mathcal{K}'}}
\newcommand{\A}{\mathcal{A}}

\addto\captionsngerman{}
\usepackage{cite}
\newcommand{\zsig}{\mathbb{Z}\left[\sigma \right]}

\newcommand{\ntp}{\mathrm{NTP}}
\newcommand{\nip}{\mathrm{NIP}}
\newcommand{\NTP}{\ntp}
\newcommand{\NIP}{\nip}
\newcommand{\tp}{{\mathrm{TP}_2}}
\newcommand{\divh}{\mathrm{div}}
\newcommand{\VF}{\mathrm{VF}}
\newcommand{\ACVF}{\mathrm{ACVF}}

\newcommand{\MODAG}{\mathrm{MODAG}}

\newcommand{\characteristic}{\mathrm{char}}

\newcommand{\Lring}{\mathcal{L}_{ring}}
\newcommand{\mycommand}

\newcounter{undefinedreferences}
\usepackage{hyperref}
\begin{document}

\maketitle
\begin{abstract}
We prove relative quantifier elimination for Pal's multiplicative valued difference fields with an added lifting map of the residue field. Furthermore, we generalize a $\NIP$ transfer result for valued fields by Jahnke and Simon to $\NTP_2$ to show that said valued difference fields are $\NTP_2$ if and only if value group and residue field are.
\end{abstract}

The model theory of valued fields with a distinguished automorphism was first studied by Belair, Macintyre, and Scanlon by adding an 
isometric automorphism in  \cite{belair2007model} and introducing the $\sigma$-hensel scheme as an analogue to henselianity in valued fields. Further progress was made by Durhan (formerly Azgin) \cite{azgin2010valued} in the contractive case. Later  Durhan and van den Dries \cite{azgin2011elementary} expanded on the case of an isometric automorphism. In 2012 Pal \cite{pal2012multiplicative}
expanded the class of admissible maps to so-called multiplicative automorphisms.
Further, Chernikov and Hils showed a $\ntp_2$-transfer result in \cite{chernikov2014valued}. We expand upon that result by adapting a  generalized $\nip$-transfer result by Jahnke and Simon \cite{jahnke2016nip}, and applying it in Pal's setting.

Let $K$ be a field and $\Gamma$ an ordered abelian group. 
A map $v:K\to \Gamma \cup \{\infty\}$ is called a valuation if for all $x,y \in K$ we have that $v(x)= \infty \Leftrightarrow x =0$,  $v(xy)=v(x)+v(y)$ and $v(x +y) \geq \min\{v(x),v(y)\}$.
We denote by $\mathcal{O}=\{ x \in K | v(x)\geq 0\}$ the valuation ring,  $ \mathfrak{m}=\{ x \in K | v(x)>0\}$ its maximal ideal,  $ \pi : \mathcal{O} \to \mathcal{O}/ \mathfrak{m}$ be the quotient map and $k:=\mathcal{O}/\mathfrak{m}$ the residue field.
\begin{ack}
Research in this paper was conducted for the author's master thesis, and the author wants to thank his master thesis advisor Martin Hils for his guidance and feedback.
\end{ack}
\section{Lifting in Valued Fields}

\begin{defi}
Let $K$ be a valued field.
A lift of the residue field (if it exists) is a field embedding  $\iota: k \to K$ satisfying $\pi(\iota(x))=x$ for all $x \in k$. This directly implies, that $v(\iota(x))=0$ for all $x \in k^\times$ and that $\characteristic(K)=\characteristic(k)$.

If $ p(x)=\sum_i a_i x^i\in k[x]$ is a polynomial, we say that the lift of $p$ along $\iota$ is given by the lift of the coefficients $p^\iota (x)= \sum_i \iota(a_i) x^i \in \mathcal{O}[x].$ 
\end{defi}
It follows from Hensel's lemma that in perfect henselian valued fields of equicharacteristic $0$ or $p$, every partial lift of the residue field can be extended to a full lift $\iota$.

\begin{defi}
Let $K$ be a valued field. 
A section of the value group is a group homomorphism $s:\Gamma \to K^\times$ such that $v(s(\gamma))=\gamma$  for all $\gamma \in \Gamma$.
\end{defi}

\begin{fact}[\mbox{\cite[Chapter 2, Corollary 28]{cherlin1976model}}]
Every $\aleph_1$-saturated valued field does admit a section.
\end{fact}

\begin{defi}
We will work in a 3-sorted language $\lann_{\Gamma,k}$ for valued fields containing the sorts $\VF$, $\Gamma$ and $k$ for the valued field, the value group and the residue field, respectively.
On the sorts $\VF$ and $k$ we have the language of rings $\lann_{Ring}$ and $\lann_{ring}$, respectively, consisting of $\{+,-,\cdot,0,1\}$ and on $\Gamma$ the language $\lann_{\Gamma}$ of ordered abelian groups $\lann_{oag}=\{+,0,<\}$ together with a constant $\infty$.
The language also contains maps for the valuation $v:\VF\to \Gamma$ and the 2 variable residue map $Res:VF\times VF \to k$ with $Res(a,b)=\pi(a/b)$ where $\pi:\VF  \to k$ is the usual residue map with $\pi(a)=0$ for $a \not \in \mathcal{O}$. Whenever we add a new function symbol $f$ to the language we will write this as $\lann_{\Gamma,k,f}$.

Often when $\K$ is a structure of a valued field, we denote the structure in the $\VF$ sort $\VF_\K$ or $K$, and the rest of the definitions above with the index $\K$, to form $\Gamma_\K$ and so on. Furthermore for a subfield $B$ of a valued field $K$ we denote $\pi(B)=:k_B$ and $v(B)=:\Gamma_B$.

Furthermore, we denote group extensions of $\Gamma$ by $\gamma$ as $\Gamma\langle \gamma \rangle $, field extensions of $K$ by an element $a$ as $K(a)$, and the extension of $\K$ as an $\lann$-structure for a language $\lann$ as $\K \langle a \rangle$.
\end{defi}

\begin{defi}\label{def ac map}
Let $\K$ be a valued field. A map $ac:K\to k_\K$ is called an angular component ($ac$-map) if $ac(x)=0 \Leftrightarrow x=0$,
 $ac|_{K^\times}:K^\times \to k^\times$ is a group homomorphism,
and  $ac(x)=\pi(x)$ for all $x \in \mathcal{O}_\K^\times$.
An $ac$-map is quantifier-free definable from a section. For $x \neq 0 $ it suffices to set  $ac(x)=Res(x,s(v(x))) $.
\end{defi} 

\begin{fact}[\cite{pas1989uniform}]
Let $T$ be the theory of henselian valued fields in equicharacteristic $0$. Then $T$ eliminates $\VF$-quantifiers in the language $\mathcal{L}_{\Gamma,k,ac}$. \label{pas theorem}
\end{fact}
\begin{prop}
\label{section embedding for residue field} \label{embedding residue}
Let $\K$ and $\K'$ be valued fields with lift (or resp. with lift and section) and $\A$ be an $\lann_{\Gamma,k,\iota}$-substructure (or resp.  $\lann_{\Gamma,k,\iota,s}$-substructure) of $\K$ and $\K'$.
Let $\alpha \in k_\K \setminus k_\A$. If there is $\beta \in k_\Kl$ such that there is an $\mathcal{L}_{ring}$-isomorphism $k_\A( \alpha ) \to k_\A ( \beta )$  sending $\alpha \mapsto \beta$ then this extends (uniquely) to an isomorphism of valued fields with lift (resp. with lift and section) $\A \langle \iota (\alpha) \rangle \to \A \langle \iota (\beta ) \rangle $ over $\A$. 
Moreover, in both caseswe have $\Gamma_{\A \langle \iota (\alpha ) \rangle}=\Gamma_\A$.
\begin{proof}
If $\alpha$ is algebraic over $k_\A$ with minimal polynomial $p(x)\in k_\A[x]$ then $\iota(\alpha)$ is also algebraic over $A=\VF_\A $ with minimal polynomial $p^\iota(x)$ by the fundamental inequality (\cite[Theorem 3.4.2] {engler2005valued}). The same holds for $\beta$ and $\iota(\beta)$. By the fundamental inequality, this extension is already unique and $\Gamma_{\A \langle \iota (\alpha ) \rangle}=\Gamma_\A$. 

If $\alpha$ is transcendental over $k_\A$, so is $\beta$. Furthermore both $\iota(\alpha)$ and $\iota(\beta)$ are transcendental over $A$ and we have Gauss extensions, which by \cite[Corollary 2.2.2] {engler2005valued} are both uniquely determined, isomorphic and preserve the value group $\Gamma_{\A \langle \iota (\alpha ) \rangle}=\Gamma_\A$.

If a section exists, it is preserved as the value group $\Gamma_\A$ remains unchanged.
\end{proof}
\end{prop}
The following is a well-known result about valued fields with a lift. We give the proof for the convenience of the reader.
\begin{prop}
Let $T$ be the theory of henselian valued fields in equicharacteristic $0$ with a lift $\iota$ of the residue field. Then $T$ eliminates $\VF$-quantifiers in the language $\mathcal{L}_{\Gamma,k,ac,\iota}$.  \label{hensel 00 qe}
\begin{proof}
Let $\mathcal{K}=(K,\Gamma_\K, k_\K)\models T$ be countable, $\K'=(K', \Gamma_\Kl, k_\Kl)\models T$ $\aleph_1$-saturated, and $\mathcal{A}=(\VF_\A ,\Gamma_\A,k_\A)\subseteq \K$ a substructure.
Furthermore let $f=(f_{\VF},f_{\Gamma}, f_k):\mathcal{A}\to \K'$ be an $\mathcal{L}_{\Gamma,k,ac,\iota}$-embedding, such that $f_\Gamma$ and $f_k$ are  $\mathcal{L}_\Gamma$- and $\mathcal{L}_{ring}$-elementary, respectively.
We need to show that we can extend $f$ to an $\mathcal{L}_{\Gamma,k,ac,\iota}$-embedding $\tilde  f:\mathcal{K}\to \K'$.
\begin{enumerate}
\item[Step 1:] We may assume that $\Gamma_\A=\Gamma_\K$, because there are no functions leaving $\Gamma$.

\item[Step 2:] We may also assume that $\VF_\A$ and $k_\A$ are subfields of $K$ and $k_\K$ respectively by passing to quotient fields.
Indeed, we can extend $f_{\VF}$ to a unique $\mathcal{L}_{Ring}$-embedding $ \tilde f_{\VF}:Q(\VF_\A)\to K'$ and
$f_k$ to a unique $\mathcal{L}_{ring}$-embedding $\tilde  f_{k}:k_{Q(\VF_\A)}\to k_\Kl$. 
The resulting map is a morphism on $k$ and $\Gamma$, because $v(a/b)=v(a)-v(b)$, $ac(a/b)=ac(a)/ac(b)$ and $\iota(\alpha/\beta)=\iota(\alpha)/\iota(\beta)$.

\item[Step 3:] 
We can extend $f$, such that $k_{\VF_\A}=k_\K$ (which implies $k_\A=k_\K$). 
Let $\alpha \in k_\K\setminus k_{\VF_\A}$, and take a $\beta\in k_\K'\setminus k_{f(\VF_\A)}$ satisfying the $\Lring$-type of $\alpha$ under $f$. Then applying Proposition \ref{embedding residue} yields a suitable extension of $f$.
\end{enumerate} 
From this point on we can conclude with Fact \ref{pas theorem}, as we are already lifting the entire residue field. So any extension of $f$ to $\K$ respects $\iota$.
\end{proof}
\end{prop}
\section{Quantifier Elimination for Valued Fields with Lift and Section}
Now we extend the previous result for the lift to include the section as well. A similar treatment of this case happens in the proof of \cite[Theorem 5.1]{vandenDries2014}, without giving the relative quantifier elimination with lift and section. For the benefit of the reader, we give our proof here and thus we need to be able to extend the value group. 
\begin{prop}\label{section qe fix}
\label{section embedding for valuegroup} \label{embedding value}
Let $\K$ and $\K'$ be valued fields with lift and section. Let $\A$ be a $\mathcal{L}_{\Gamma,k,\iota,s}$-substructure of $\K$ and $\K'$. Let $\gamma \in \Gamma_\K \setminus \Gamma_{\A}$.  If there is $\gamma' \in \Gamma_\Kl$ such that $\Gamma_\A \langle \gamma \rangle \to \Gamma_\A \langle \gamma ' \rangle$ is an ordered group isomorphism,
then there is an isomorphism $\A\langle s(\gamma)\rangle\to \A \langle s (\gamma')\rangle$ of valued fields with lift and section over $\A$ sending $s(\gamma)$ to $s(\gamma')$. Moreover, we have $k_{\A\langle s(\gamma) \rangle}=k_\A$.
\begin{proof}
To show this we verify that the extension $\A\langle s(\gamma)\rangle$ is unique up to isomorphism.

If $\gamma \in \divh(\Gamma_\A)$ then there is a minimal $n\in \mathbb{N}$ such that $n \gamma \in \Gamma_\A$. Then $s(\gamma)$ is also an $n$th root of $s(n\gamma)$ in $A$. By the fundamental inequality, the extension is already unique up to isomorphism and the minimal polynomial is given by $x^n -s(n\gamma)$. The same holds for $\gamma'$, so the extensions $\A\langle s( \gamma) \rangle$ and $\A \langle s(\gamma')\rangle $ have to be isomorphic. 

If $\gamma \not \in \divh(\Gamma_\A)$ then $s(\gamma)$ is transcendental over $A$.  This already determines the valued field extension uniquely by \cite[Theorem 3.4.2]{engler2005valued}.  The same is again true for $s(\gamma')$, yielding the isomorphism.

By the fundamental inequality for $\gamma \in \divh(\Gamma_\A)$ and \cite[Theorem 3.4.2]{engler2005valued} for $\gamma \not \in \divh(\Gamma_\A)$ we can verify that $k_{\A\langle s(\gamma) \rangle}=k_\A$.
\end{proof}
\end{prop}

\begin{thm}\label{qe for lift section pas}
The theory of henselian valued fields in equicharacteristic 0 with lift and section eliminates $VF$-quantifiers in $\lann_{\Gamma,k,\iota,s}$.
\begin{proof}
Apply Proposition \ref{section qe fix} in Step 1 and conclude as in Proposition \ref{hensel 00 qe}.
\end{proof}
\end{thm}
\begin{rem}
A version of this proof in positive characteristic also works in algebraically maximal Kaplansky fields, see \cite{MR1703196}.
\end{rem}

\begin{fact} \label{immediate is NIP section}

Note that for immediate extension $M(b)/M$ of a model of $T$ the type $tp_{\lann_{\Gamma, k,\iota,s}}(b/M)$ is already fully implied by $qftp_{\lan_{\Gamma,k}}(b/A)$, which consists of instances of $\NIP$ formulas as $\ACVF$ is $\NIP$. 
\end{fact}
Remarkably we maintain stable embeddedness in both cases.

\begin{prop}
In a model of T, the residue field and the value group are purely stably embedded and orthogonal in $\mathcal{L}_{\Gamma,k,\iota,s}$. \label{stab embedded}
 \begin{proof}

 Let $\mathcal{M} \prec \mathcal{N}$ be models of $T$ and $\boldsymbol \gamma \in \Gamma_\cN^n$ and $\boldsymbol \delta \in \Gamma_\cN^n $ tuples with $tp_{\lann_{\Gamma}}(\boldsymbol \gamma/\Gamma_\cM)= tp_{\lann_{\Gamma}} ( \boldsymbol \delta/\Gamma_\cM)$. Then we have  $\langle M(s(\boldsymbol \gamma)), \Gamma_\cM \langle \boldsymbol \gamma \rangle, k_\cM\rangle \iso \langle N(s(\boldsymbol \delta)), \Gamma_\cM \langle \boldsymbol \delta \rangle , k_\cM\rangle$ by iterating Proposition \ref{section embedding for residue field}.
 
Similarly, for $\boldsymbol \alpha  \in k_\cN^n$ and $\boldsymbol \beta \in k_\cN^n $ tuples with $tp_{\lann_{k}}(\boldsymbol \alpha/k_\cM)= tp_{\lann_{k}} ( \boldsymbol \beta/k_\cM)$ iterating Proposition \ref{section embedding for valuegroup} yields $\langle M(\iota(\boldsymbol \alpha)), \Gamma_\cM , k_\cM (\boldsymbol \alpha )\rangle \iso \langle N(\iota(\boldsymbol \beta)), \Gamma_\cM  , k_\cM( \boldsymbol \beta )\rangle$. This yields stable embeddedness. 

For purity consider $\boldsymbol \gamma, \boldsymbol \delta \in \Gamma^n$ with $\boldsymbol\gamma \equiv_{\lann_\Gamma} \boldsymbol\delta$. Let $\mathbb{Q}$ be the prime subfield of $K$. 
Now consider the generated $\lann_{\Gamma,k,\iota,s}$-structures $\langle \mathbb{Q}(s(\boldsymbol\gamma)), \mathbb{Z} \cdot \boldsymbol\gamma,\mathbb{Q} \rangle $ and $\langle \mathbb{Q}(s(\boldsymbol\delta)),\mathbb{Z} \cdot \boldsymbol\delta, \mathbb{Q}\rangle $. 
Then by \ref{section embedding for valuegroup} the isomorphism $ \mathbb{Z} \cdot\boldsymbol \gamma \to  \mathbb{Z} \cdot \boldsymbol\delta$ sending $\boldsymbol\gamma$ to $\boldsymbol\delta$ extends to one of the whole structure.
 By $\VF$-quantifier elimination it is already elementary. The argument for $k$ follows analogously from Proposition \ref{section embedding for residue field}. The same proof also shows that $k \cup \Gamma$ is stably embedded and pure.

For orthogonality let $\boldsymbol \alpha =(\alpha_0,\dots,\alpha_{n-1})\in k^n$ and $ \boldsymbol \gamma = (\gamma_0,\dots ,\gamma_{m-1})\in \Gamma^m$. We want to show that $tp(\boldsymbol \alpha) \cup tp(\boldsymbol \gamma) \vdash tp(\boldsymbol \alpha, \boldsymbol \gamma)$. 
So take $\beta_0,\dots,\beta_{n-1} \equiv \alpha_0,\dots,\alpha_{n-1}$ and $\delta_{0},\dots,\delta_{m-1} \equiv \gamma_0,\dots,\gamma_{m-1}$. Then by iterating Proposition \ref{section embedding for residue field} and Proposition \ref{section embedding for valuegroup} we get that the generated structures
$$\langle \mathbb{Q}(\iota(\alpha_i),s(\gamma_j) : i<n, j< m),\langle \gamma_j : j<m\rangle ,\mathbb{Q}(\alpha_i : i < n) \rangle$$ 
and 
$$\langle \mathbb{Q}(\iota(\beta_i),s(\delta_j) : i<n, j< m),\langle \delta_j : j<m\rangle ,\mathbb{Q}( \beta_i : i < n) \rangle$$
are isomorphic. By $\VF$-quantifier elimination, the isomorphism is already elementary. Then $tp(\boldsymbol \alpha,\boldsymbol \gamma)$ is already fully implied.
\end{proof}
\end{prop}

\section{Multiplicative Ordered Abelian Groups}
Let $G$ be an ordered abelian group with an automorphism $\sigma$.
Then we can write $G$ as a $\mathbb{Z}[\sigma]$-module through the action on $G$ by $L(x)=m_k\sigma^k(x)+\dots+m_0 x, m_i \in \mathbb{Z}$. 

We restrict to the cases where $G$ is an orderd $\mathbb{Z}[\sigma]$-module i.e. for all $L(x)$ we have 
 $$(\forall x>0: L(x)>0) \text{ or } (\forall x>0 :L(x)=0) \text{ or } (\forall x>0 :L(x)<0).$$
 Once we identify the  $L(x) \in \mathbb{Z}[\sigma]$ inducing the same action on $G$, we get an ordered ring $\mathbb{Z}[\sigma]/_\approx =:\mathbb{Z}[\rho]$ where we denote the quotient map by $\Phi$.
 
We can embed this ring into a real closed field and write the action of $\sigma$ as multiplication with $\rho$.
We define the set of $\zsig$-positives of $G$ as $$ptp_{\zsig}(G)=\{L\in \zsig : \forall x \in G (x>0 \Rightarrow L(x) > 0)\}.$$
We say $G$ and $G'$ have the same $\rho$ if their sets of $\zsig$-positives coincide.

\begin{defi}
$\MODAG_\rho$ is the theory of the class of non-trivial multiplicative ordered difference groups of a given $\rho$. We will call it $\MODAG$ for arbitrary $\rho$. For $H$ a substructure of  $G \models\MODAG_\rho$ and $\gamma \in G$ we write $H\langle\gamma\rangle$ for the generated ordered difference group of $\gamma$ over $H$ in $G$. In particular $H\langle \gamma \rangle$ is closed under $\rho^{-1}$.
\end{defi}
 
\begin{defi}
Let $G$ be a $\MODAG_\rho$. It is called divisible, if for each  $L\in\zsig$ with $\Phi(L)\neq 0$ and $b \in G$   the equation $L(x)=b$ has a solution.  We denote the theory of non-trivial divisible models of $\MODAG_\rho$ as $\text{div-}\MODAG_\rho$.
\end{defi}
\begin{fact}[\cite{pal2012multiplicative}]
The theory $\text{div-MODAG}_\rho$ has quantifier elimination and is the model companion of $MODAG_\rho$. Furthermore, $\MODAG_\rho$ is o-minimal and therefore $\NIP$.
\end{fact}

\section{Valued Difference Fields}
\begin{defi}
A valued difference field in equicharacteristic $0$ is a valued field $\K=(K,\Gamma,k;v,\pi)$   with a distinguished automorphism $\sigma$ of $K$ which also satisfies $\sigma(\mathcal{O}_\K)=\mathcal{O}_\K$.
Consequently, $\sigma$ induces an automorphism of the residue field $\bar \sigma:k\to k$ via
$$\pi (a) \mapsto \pi(\sigma(a))\text{, for all } a\in \mathcal{O}_\K.$$
Equipping $k$  with $\bar \sigma$ yields the structure of a difference field.
Likewise $\sigma$ also induces an automorphism $\sigma_\Gamma: \Gamma \to \Gamma$ on the value group via
$$\gamma \mapsto v(\sigma(a))\text{ for some $a$ with } v(a)=\gamma.$$
We will also write $\sigma_\Gamma$ as $\sigma$ if it is clear from the context where it operates.  

We denote difference field extensions by $k \langle a \rangle$.
With $\sigma_\Gamma$ the value group has the structure of an ordered abelian difference group. 

For our purposes, we will define the languages $\lann_{ring,\sigma}=\lann_{ring}\cup\{\bar \sigma\}$, $\lann_{\Gamma,\sigma}=\lann_{\Gamma}\cup\{\sigma_{\Gamma}\}$ on their respective sorts and $\lann_{\Gamma,k,\sigma}$ as the extension of $\lann_{\Gamma,k}$ by $\sigma$ on all sorts.

For difference fields $K\subseteq K'$ and $a \in K' \setminus K$ we will write $K\langle a \rangle$ for the inversive difference ring generated by $a$ over $K$. In particular $K\langle a \rangle$ is closed under $\sigma^{-1}$. This is a slight abuse of notation as $\sigma^{-1}$ is not part of the language, but as $\sigma$ is an automorphism of the ambient field this closure is well defined.
\end{defi}

To each polynomial $F(x_0,\dots,x_{n})\in K[x_0,\dots,x_n]$ we assign a $\sigma$-polynomial $P(x)$ in one variable given by $P(x)=F(x,\sigma(x),\dots,\sigma^n(x))$.   We will often write this as $P(x)=F(\boldsymbol{\sigma}(x))$. The highest $d$ such that some coefficient of some monomial containing $ \sigma^d(x) $ is non-zero is called the order of $P$.

Let $\boldsymbol a=(a_0,\dots,a_n)$ be a tuple of field elements and $\boldsymbol I=(i_0,\dots,i_n)\in \mathbb{N}^{n+1}$ a multi-index and define $\boldsymbol a^{\boldsymbol I}=a_0^{i_0}\dots a_{n}^{i_{n}}$.
We will denote the index with $1$ in the $i$th component and 0 everywhere else as $e_i$. 

On $\mathbb{N}^{n+1}$ we establish $I\leq J$ as the  partial product ordering induced by $\mathbb{N}$, and $I < J$ for $I \leq J$ and $I \neq J$. 

Let $\tau$ be any element from any ring.
We define the $\tau$-length of $\boldsymbol I$ as $|\boldsymbol I|_\tau=i_0\tau^0+i_1\tau^1+\dots+i_n\tau^n$ and the length of $\boldsymbol I$ as $|\boldsymbol I|:=|\boldsymbol I|_1$.
Then $|\boldsymbol I|\in \mathbb{N}$ and $|\boldsymbol I|_\tau$ is in the ring $\tau$ stems from.  
Motivated by Taylor expansions, we define for tuples of indeterminates $\boldsymbol x=(x_0,\dots,x_n)$, $\boldsymbol y=(y_0,\dots,y_n)$ and any polynomial $F(\boldsymbol x)$ over $K$ of characteristic 0, the polynomial
$$F_{(\boldsymbol I)}(\boldsymbol x)=\left(\dfrac{\partial}{\partial x_0}\right)^{i_0}\dots\left(\dfrac{\partial}{\partial x_n}\right)^{i_n} \dfrac{F( \boldsymbol x)}{i_0!\dots i_n!},$$
and we write $P_{(\boldsymbol  I)} (x)$ for the polynomial if  $P_{(\boldsymbol  I)} (x) =F_{(\boldsymbol I )} (\boldsymbol \sigma (x))$.

\begin{defi}
We say $k$ is linear difference closed, if  or all $\alpha_0,\dots,\alpha_n \in k$ not all $0$ there exists a solution in $k$ for
\begin{equation*}
1+\alpha_0 x +\alpha_1  \bar\sigma (x) + \dots + \alpha_n \bar \sigma ^n(x)=0.
\end{equation*}
\end{defi}

\begin{defi}
Let $\mathcal{K}=(K,\Gamma, k;v,s,\pi,\iota,\sigma)$ be a valued difference field. $\mathcal{K}$ is called $\sigma$-multiplicative if  $\Gamma $ is a $MODAG$ with $\rho \geq 1$.
\end{defi}

\subsection{$\sigma$-henselianity}

Following Pal \cite[Definition 5.1]{pal2012multiplicative}, we introduce an analogous property to henselianity for the difference valued field case.
\begin{defi}
For a $\sigma$-polynomial $P(x)$ over $K$ of order $\leq n$ and $a \in K$, we say that $(P,a)$ is in $\sigma$-henselian configuration, if $P$ is non-constant and there are  $0 \leq i\leq n$ and a $\gamma \in \Gamma$ such that:
\begin{itemize}
\item $v(P(a))=v(P_{(e_i)}(a))+\rho^i \gamma \leq v( P_{(e_j)}(a))+\rho^j\gamma$ for all $0\leq j\leq n$,
\item $v(P_{(\boldsymbol J)}(a))+|\boldsymbol J|_\rho\gamma < v(P_{(\boldsymbol L)}(a))+|\boldsymbol L|_\rho\gamma$ for all $0\neq \boldsymbol J<\boldsymbol L\in \mathbb{N}^{n+1}$ and $P_{(\boldsymbol J)}\neq 0$. 
\end{itemize}
In particular if $(P,a)$ is in $\sigma$-henselian configuration we set $\gamma(P,a)=\gamma$ for the unique $\gamma$ above.
\end{defi}

In analogy to henselian valued fields, we define:

\begin{defi}
A multiplicative valued difference field is $\K$ is called $\sigma$-henselian if for all $(P,a)$ in $\sigma$-henselian configuration there is some $b \in K$ such that $v(b-a)=\gamma(P,a)$ and $P(b)=0$. 
\end{defi}
It is a fairly natural consideration to add a lifting map to $\sigma$-henselian valued difference fields.
\begin{fact}[{\cite{pal2012multiplicative}}]\label{existence of lift}
Let $\K$ be $\sigma$-henselian without a lift and $K_0\subseteq \mathcal{O}_\K$ be a $\sigma$-subfield of $K$. Then there is a $\sigma$-subfield $K_0\subseteq K_1\subseteq \mathcal{O}_\K$ with $\pi(K_1)=k$.
\end{fact}

\begin{fact}[{\cite[Corollary 5.9]{pal2012multiplicative}}]\label{fact:ps+ldc=sh}
A pseudo-complete $\sigma$-multiplicative valued difference field with linear difference closed residue field is $\sigma$-henselian.
\end{fact}

\begin{rem}
For an ordered abelian group $\Gamma$ and a field $k$ the Hahn series field of $k$ and $\Gamma$ is given by $$k((t^\Gamma))=\left\lbrace  a= \sum_{\gamma \in \Gamma} a_\gamma t^\gamma \phantom{x} \vline \phantom{x} a_\gamma\in k  , supp(a)=\{\gamma \in \Gamma : a_\gamma \neq 0\} \text{ is well-ordered} \right\rbrace.  $$  
Then a lift is given by $a \mapsto a  t^0$ and a section by $\gamma \mapsto t^\gamma$.
Furthermore $k((t^\Gamma))$ is pseudo-complete.

In particular, if $\Gamma \models \MODAG$ and $k$ is a linear difference closed difference field, then there is an induced $\sigma$-equivariant automorphism on $k((t^\Gamma))$ defined by
$$\sigma \left( \sum_{\gamma \in \Gamma} a_\gamma t^\gamma \right) = \sum_{\gamma \in \Gamma} \bar \sigma(a_\gamma) t^{\sigma_\Gamma(\gamma)}.$$ 
By Fact \ref{fact:ps+ldc=sh}, $k((t^\Gamma))$ is a $\sigma$-henselian multiplicative valued difference field with lift and ($\sigma$-equivariant) section.
\end{rem}

\subsection{Quantifier Elimination}
To prove quantifier elimination we will again need to extend the value group and the residue field. To avoid confusion, we will now denote the non-difference valued field extensions of a structure $\A$ by an element $a$ as $\A (a)$ in contrast to the valued difference field extensions $\A \langle a \rangle$.

We will assume by convention that all lifts and sections of valued difference fields are $\sigma$-equivariant.
\begin{lem}\label{extending residue field 2} 
Let $\K$ and $\K'$ be multiplicative valued difference fields with lift and section. 
Let $\A$ be a substructure of $\K$ and $\K'$. Let $\alpha \in k_\K \setminus k_{\A}$.  If there is a $\beta \in k_\Kl$ such that $k_\A \langle \alpha \rangle \to k_\A \langle \beta \rangle$ is a difference field isomorphism,  
 then there is an isomorphism of valued difference fields $ \A \langle \iota(\alpha) \rangle \to \A \langle \iota(\beta) \rangle$ over $\A$ sending $\iota(\alpha) \mapsto \iota(\beta)$. Furthermore for the value group it holds that $\Gamma_{\A\langle \iota(\alpha) \rangle}=\Gamma_\A$.
\begin{proof}
We have to verify that the extension $\A\langle \iota (\alpha)\rangle$ as a valued difference field is uniquely determined.

This is the case as by Proposition \ref{section embedding for residue field} the valued field extension $\A(\iota(\alpha))$  is unique as a valued field with lift and section and by iterating we get $\A(\sigma^n(\iota(\alpha)) : n \in \mathbb{N})$ uniquely determined as a valued difference field. As we also have the lift, the automorphism is uniquely determined as $\sigma(\iota(\alpha))=\iota(\sigma(\alpha))$. The same is true for $\iota(\beta)$.
Applying $\sigma^{-k}$ for all $k\in \mathbb{N}$ yields the desired map.
Because the value group doesn't change when applying Proposition \ref{section embedding for residue field}, $\Gamma_{\A\langle \iota(\alpha)\rangle}=\Gamma_\A$ follows trivially.
\end{proof}
\end{lem}

Similarly, we can show for the value group:
\begin{lem}\label{extending value group}
Let $\K$ and $\K'$ be multiplicative valued difference fields with lift and section. 
Let $\A$ be a substructure of $\K$ and $\K'$. Let $\gamma \in \Gamma_\K \setminus \Gamma_{\A}$.  If there is a $\gamma' \in \Gamma_\Kl$ such that $\Gamma_\A \langle \gamma \rangle \to \Gamma_\A \langle \gamma ' \rangle$ is a MODAG isomorphism,  
 then there is an isomorphism of valued difference fields $ \A \langle s(\gamma) \rangle \to \A \langle s(\gamma') \rangle$ over $\A$ sending $s(\gamma) \mapsto s(\gamma')$. Furthermore, for the residue field, it holds that $k_{\A\langle \gamma\rangle}=k_\A$.
\begin{proof}
We again have to verify that the extension of valued difference fields $\A\langle s(\gamma)\rangle$ is uniquely determined.

This is the case as by Proposition \ref{section embedding for valuegroup} the valued field extension $\A(s(\gamma))$ is unique as a valued field with lift and section and by iterating we get $\A(\sigma^n(s(\gamma)) : n \in \mathbb{N})$ uniquely determined as a valued difference field. As we also have the section, the automorphism is uniquely determined as $\sigma(s(\gamma))=s(\sigma(\gamma))$. The same is true for $s(\gamma')$.
Applying $\sigma^{-k}$ for all $k\in \mathbb{N}$ yields the desired map.
Because the residue field doesn't change when applying Proposition \ref{section embedding for valuegroup}, $k_{\A\langle s( \gamma)\rangle}=k_\A$ follows trivially.
\end{proof}
\end{lem}

Now we can prove quantifier elimination.

\begin{thm}
The theory of $\sigma$-henselian multiplicative valued difference fields with $\sigma$-equivariant lift and section eliminates ${\VF}$-quantifiers in the language $\mathcal{L}_{\Gamma, k, \sigma, \iota, s}$.
\begin{proof}
Let $\mathcal{K}$ be a model of size $\aleph_0$ with a substructure $\mathcal{A}\subseteq \mathcal{K}$.
Let $\Kl$ be an $\aleph_1$-saturated model and $f:\mathcal{A}\to \Kl$ an embedding, such that $f_\Gamma$ and $f_k$ are elementary in $\Gamma$ and $k$.
We will now extend $f$ to an embedding $\tilde f :\mathcal{K}\to \Kl$.
\begin{enumerate}
\item[Step 1:] By passing to the fraction field and closing under lift and section we can assume that $\VF_\A$ and $k_{A}=k_\A$ are fields and $\Gamma_{\VF_\A}=\Gamma_\A \models \text{OAG}$.
Furthermore, we can close the structure under $\sigma^{-1}$.
\item[Step 2:] We may achieve $k_{\VF_\A}=k_\K$. 

Let $\alpha \in k_\K \setminus k_{\VF_\A}$.
Take $\beta\in k_\Kl$ with the same $\lann_{ring,\sigma}$-type as $\alpha$ over $k_{\VF_\A}$ under $f$. 
Then there is an elementary difference field isomorphism $k_{\VF_\A} \langle \alpha \rangle \to f_k(k_{\VF_\A})\langle \beta \rangle$ extending $f_k$.
By Lemma \ref{extending residue field 2} there exists a suitable unique extension of $f$ with domain $\langle \VF_\A\langle \iota(\alpha)\rangle,\Gamma_{\VF_\A},k_{\VF_\A}\langle \alpha \rangle \rangle$ sending $\alpha $ to $\beta$, that is elementary in $\Gamma$ and $k$.

\item[Step 3:] We may achieve $\Gamma_{\VF_\A} =\Gamma_\K$.

Let $\gamma \in \Gamma_\K \setminus \Gamma_{\VF_\A}$.
Find a $\gamma'\in\Gamma_\Kl$ satisfying the same $\mathcal{L}_{\Gamma,\sigma}$-type as $\gamma $ over the image of $\Gamma_\A$ under $f_\Gamma$. Then we get an elementary $\MODAG$-isomorphism  $\Gamma_{\VF_\A}\langle \gamma \rangle \to f_\Gamma(\Gamma_{\VF_\A})\langle \gamma' \rangle $. By Lemma \ref{extending value group} we can find 
an extension  $f':\A\langle s(\gamma) \rangle \to \A\langle s(\gamma') \rangle \subseteq \Kl$ of $f$ over $\A$ that is sending $s(\gamma)$ to $s(\gamma')$ and is elementary in $\Gamma$ and $k$.

\end{enumerate}
From this point onwards we can conclude the proof with the quantifier elimination result by Pal without the lift of the residue field \cite[Theorem 11.8]{pal2012multiplicative}.
\end{proof}
\end{thm}

\begin{rem}
Just as in the setting of Pal in  \cite{pal2012multiplicative} this proof yields an Ax-Kochen-Ershov-Principle and the decidability of the theory $T_{\sigma,\iota,s}$ if residue field and value group are decidable. 
\end{rem}
Furthermore, in both our result above  and \cite[Theorem 11.8]{pal2012multiplicative}  it  suffices  to work with an angular component map instead of a section in order to achieve quantifier elimination, hence we state:
\begin{thm}
The theory of $\sigma$-henselian multiplicative valued difference fields with $\sigma$-equivariant lift and angular component eliminates ${\VF}$-quantifiers in the language $\mathcal{L}_{\Gamma, k, \sigma, \iota, ac}$.
\end{thm}

In addition, we also obtain similarly to \cite{pal2012multiplicative}:
\begin{cor}
The model companion of multiplicative valued difference fields in equicharacteristic 0 with a $\sigma$-equivariant section and a $\sigma$-equivariant lift is the theory of $\sigma$-henselian multiplicative valued difference fields  with lift and section, where the value group is a div-$\MODAG$ and the residue field is a model of $\text{ACFA}$. 
\end{cor}
To apply the $\ntp_2$-transfer theorem in the next section we need the following. 
\begin{prop}
In a $\sigma$-henselian multiplicative valued difference fields with $\sigma$-equivariant lift and section
the residue field and value group are purely stably embedded and orthogonal. 
\begin{proof}
The proof follows exactly as for Proposition \ref{stab embedded} with a few minor substitutions. The residue field extension from Proposition \ref{section embedding for residue field} needs to be replaced with its difference counterpart in Lemma \ref{extending residue field 2} and similarly for the value group with \ref{section embedding for valuegroup} and Lemma \ref{extending value group}. 
\end{proof}
\end{prop}

\begin{prop}
Let $K$ be a $\sigma$-henselian multiplicative valued difference field with $\sigma$-equivariant lift and section and $a$ a singleton from the $\VF$-sort. 
If $K\langle a \rangle/K$ is an immediate extension, then every formula in $tp_{\lann_{\Gamma,k,\sigma,\iota,s}}(a/K)$ is determined by $qftp_{\lann_{\Gamma,k,\sigma}}(a/K)$, which is determined by instances of  $\NIP$ formulas.
\begin{proof}
By quantifier elimination the type $tp_{\lann_{\Gamma,k,\sigma,\iota,s}}(a/K)$ is implied by \linebreak $qftp_{\lann_{\Gamma,k,\sigma}}(a/K)$.

Suppose a formula $\phi \in qftp_{\lann_{\Gamma,k,\sigma}}(a/K)$ is $\text{IP}$.
Any formula in $qftp_{\lann_{\Gamma,k,\sigma}}(a/K) $ is equivalent to boolean combination of valuative conditions on $\sigma$-polynomials. 

 For a $\sigma$-polynomial $P(x)=F(\boldsymbol \sigma(x))$ with $F\in K[X_0,\dots,X_n]$ however we can replace any witness $b\in K$ with the tuple $\boldsymbol \sigma(b)$. Then we get an $\text{IP}$ formula $\psi \in qftp_{\lann_{\Gamma,k}}((x_0,\dots,x_n)/K)$. But all formulas in this type are instances of $\NIP$ formulas as $\ACVF$ is $\NIP$, a contradiction.
\end{proof}
\end{prop}

\section{$\NTP_2$ Transfer}

Jahnke and Simon have proven the following $\NIP$ transfer result, which we will adapt to $\ntp_2$.
\begin{defi}[\cite{jahnke2016nip}]
For a complete theory $T$ of valued fields possibly with additional structure, we define the following properties:
\begin{itemize}
\item[(SE):] The residue field and the value group are stably embedded.
\item[(IM):] If $K\models T$ and $a$ is a singleton in an extension of $K$ such that the generated structure $K\langle a\rangle/K$ is an immediate extension, then $tp(a/K)$ is implied by instances of $\NIP$ formulas. 
\end{itemize}  
\end{defi}
\begin{thm}[{\cite[Theorem 2.3]{jahnke2016nip}}]
Under the assumptions (SE) and (IM), a complete theory of valued fields $T$ is $\NIP$ if and only if the theories of the residue field and the value group are.
\label{jahnke simon NIP transfer}
\end{thm}
This already yields:
\begin{cor}
The theory of a henselian valued field of equicharacteristic 0 with lift and section is $\NIP$, if and only if the residue field is $\NIP$.
\end{cor}

It has been remarked in \cite{jahnke2016nip} that adapting this to $\ntp_2$ using \cite{chernikov2014valued} is possible.
We adapt (IM) to work in the $\ntp_2$ context:
Again, for a complete theory $T$ of valued fields with possible additional structure, we define:
\begin{itemize}
\item[(IM'):] If $K\models T$ and $a$ is a singleton from an extension such that the generated structure $K\langle a\rangle /K$ is an immediate extension, then every formula in $tp(a/K)$ is equivalent to an instance of an $\text{NTP}_2$ formula. 
\end{itemize} 

For this definition to make sense, we should give a definition for $\ntp_2$.
\begin{defi}\label{def ntp2 and strong}
An inp-pattern is an array $\{\phi_i(x,y_i),(a_{ij})_{j<\omega},k_i\}_{i< \lambda}$ consisting of formulas $\phi_i (x,y_i)$, $a_{ij}$ tuples of parameters of length $|y_i|$ and $k_i\in \mathbb{N}$,  for some cardinal $\lambda$, which we call the depth, such that
\begin{itemize}
\item $\{ \phi_i(x,a_{ij})\}_{j<\omega}$ is $k_i$-inconsistent for every $i<\lambda$,
\item $\{\phi_i(x,a_{if(i)})\}_{i<\lambda}$ is consistent for every $f:\lambda \to \omega$.
\end{itemize}
The burden of a theory $T$ is the supremum of the depths of inp-patterns in all models of $T$.
$T$ is said to be strong if there is no inp-pattern of infinite depth in $T$.
A formula $\phi(x,y)$ is $\tp$ if there is an inp-pattern of the form \linebreak $\{\phi(x,y),(a_{ij})_{j<\omega},k_i\}_{i< \omega}$. Otherwise $\phi$ is called $\NTP_2$.
A theory is called $\NTP_2$ if no formula has $\tp$.

\end{defi}

\begin{prop}
The theory of a non-trivially valued field $\K$ with a lift is not strong.
\end{prop}
\begin{proof}
Let $t \in m_\K$ be non-zero and for $n< \omega$, let $$X_n=\{x\in K : \exists y_0,...,y_n\in k \text{ such that } v(x-\sum_{k=0}^n\iota(y_k)t^k)\geq (n+1)v(t)\}.$$
Then there is an obvious surjective map $f_n=(p_0,...,p_n):X_n\to k^{n+1}$. We have for any $n$, that
$X_{n+1}\subseteq X_n$ and that $f_n$ restricted to $X_{n+1}$ equals $f_{n+1}$ followed by the projection from $n+2$ coordinates on the
first $n+1$ coordinates. In this way, by setting $X=\bigcap_{n< \omega} X_n$ we get a surjective pro-definable map
$f=(p_n: n< \omega):X \to k^\omega$.

Now choose any array $(a_{ij})_{i,j< \omega}$ of pairwise distinct elements of the residue field $k$, and let
$\phi_n(x,y):=x\in X_n\wedge p_n(x)=y$. By construction, this is an inp-pattern, showing that $Th(\K)$ is not strong.
\end{proof}
Similarly, for the section case, we can show: 
\begin{prop}
The theory of a non-trivially valued field $\K$ with an infinite residue field and a section is not strong.
\end{prop}
\begin{proof}
Let $g(x)=x-s(v(x))$ and
 $$X_n=\{x\in K : \exists \gamma_0<...<\gamma_n\in \Gamma \text{ such that } v(\underbrace{g \circ \dots \circ g}_{i \text{ times}}(x))=  \gamma_i  \quad \forall i \leq n\}.$$
Then as above there is a surjective map $f_n=(p_0,...,p_n):X_n\to \Gamma^{n+1}_<$, where $\Gamma^{n+1}_<$ is the set of properly increasing $n+1$ tuples in $\Gamma$. We again have for any $n$, that  $X_{n+1}\subseteq X_n$ and that $f_n$ restricted to $X_{n+1}$ equals $f_{n+1}$ followed by the projection from $n+2$ coordinates on the
first $n+1$ coordinates. In this way, by setting $X=\bigcap_{n< \omega} X_n$ we get a surjective pro-definable map
$f=(p_n: n< \omega):X \to \Gamma^\omega_<$.

Now choose any array $(\gamma_{i,j})_{i,j< \omega}$ of pairwise distinct elements of value group $\Gamma$ such that $\gamma_{i,j_1}<\gamma_{i+1,j_2}$ for all $i,j_1,j_2<\omega$, and let
$\phi_n(x,y):=x\in X_n\wedge p_n(x)=y$. We have again constructed an inp-pattern of depth $\aleph_0$, showing that $Th(\K)$ is not strong.
\end{proof}

For our means, it is sufficient to look at singletons because of the following fact:
\begin{fact}[\cite{chernikov2014theories}]
If $T$ has $\tp$, then there is already some $\phi(x,y)$ with $|x|=1$ that has $\tp$.
\end{fact}
As we are working with arrays instead of sequences, we have to introduce more terminology.
\begin{defi}
We say that an array $(a_{ij})_{i<\kappa,j<\omega}$ for some cardinal $\kappa$ is mutually indiscernible if for all $i$ the sequence $(a_{ij})_{j<\omega}$ is indiscernible over $\{a_{kj}: k\neq i, j < \omega\}$.

If additionally the sequence $((a_{ij})_{j<\omega})_{i<\kappa}$ is  indiscernible, we say $(a_{ij})_{i<\kappa,j<\omega}$ is a strongly indiscernible array.
\end{defi}

For our endeavour, we will need the Array Extension Lemma by Chernikov and Hils.
\begin{lem}[{\cite[Lemma 3.8]{chernikov2014valued}}]

\label{chernkiov hils 3.8}
Let $D$ be a stably embedded $\emptyset$-definable set and assume that the induced structure $D_{ind}$ on $D$ is $\text{NTP}_2$.
Further assume that $( (c_{ij})_{j<\omega})_{i < \omega}$ is indiscernible over $a$ and that $(c_{ij})_{i,j< \omega}$ is a strongly indiscernible array.

Let a small $b \subseteq D$ be given. 
Then there is a $(c_{ij}^*)_{i,j< \omega}$ and $(b_{ij}^*)_{i,j< \omega}$ such that:
\begin{enumerate}
\item $( b_{ij}^* c_{ij}^*)$ as a sequence in $i$ is indiscernible over $a$,
\item $( b_{ij}^* c_{ij}^*)_{j<\omega}$ are mutually indiscernible,
\item $ (c _{ij}^*)_{j<\omega} \equiv _{c_{i0}}  (c_{ij})_{j <\omega}$ for all $i < \omega$,
\item $abc_{00}\equiv ab^*_{00}c^*_{00}$
\end{enumerate}
In particular $(b_{ij}^*c_{ij}^*)_{i,j < \omega}$ is a strongly indiscernible array.
\end{lem}
Now that we have all the ingredients collected, we can prove using the strategy of Jahnke and Simon a more general version of \cite[Theorem 4.1]{chernikov2014valued}.
\begin{thm} \label{ntp2 transfer}
Under the assumptions (SE) and (IM'), the theory $T$ of a valued field (possibly with extra structure) is $\text{NTP}_2$ if and only if the full induced theories on the residue field and on the value group are $\NTP_2$.
\begin{proof}
Assume $T$ is has $\mathrm{TP}_2$. 
Then there is a formula $\phi(x,y)$ with $x$ being a single variable and a strongly indiscernible array $(a_{ij})_{i,j< \omega}$ witnessing $\tp$ together with a singleton $a\models \{ \phi (x,a_{i0})\}_{i < \omega}$ such that the sequence of rows $((a_{ij})_{j < \omega})_{i < \omega}$ is indiscernible over $a$.
We may increase each $a_{ij}$ so that it enumerates a model $M_{ij}$.
Then let $b$ and $c$ be enumerations of the residue field of the generated structure $M_{00}\langle a\rangle $ and of the value group of $M_{00}\langle a\rangle $ respectively.

Using the Array Extension Lemma \ref{chernkiov hils 3.8} twice, we can find $(a_{ij}^*)_{i,j < \omega}$, $(b_{ij}^*)_{i,j < \omega}$ and $(c_{ij}^*)_{i,j < \omega}$ such that
\begin{enumerate}
\item $( c_{ij}^*  b_{ij}^*  a_{ij}^*)$ is indiscernible in $i$ over $a$,
\item $( c_{ij}^* b_{ij}^* a_{ij}^*)$ are mutually indiscernible,
\item $ (a_{ij}^*)_{j<\omega}\equiv_{a_{i0}}  (a_{ij})_{j<\omega}$ for all $i <  \omega$,
\item $ac_{00}b_{00}a_{00}\equiv a c^*_{00} b^*_{00} a^*_{00}$.
\end{enumerate}
Then we can extend every tuple $c_{ij}^*b_{ij}^*a_{ij}^*$ to a model.

By iterating the array extension procedure $\omega$  times, we get a strongly indiscernible array of models 
$N_{ij}$ for all $i,j< \omega$, where $N_{ij}\langle a\rangle /N_{ij}$ is immediate for all $i,j< \omega$.
By construction, the array $(N_{ij})_{i,j < \omega}$ witnesses $\text{TP}_2$ in $\phi$.
Now by assumption (IM') every formula in the type $tp(a/N_{00})$ is equivalent to an instance of an $\text{NTP}_2$ formula. As none of these can be equivalent to our altered $\phi$, we have a contradiction.
\end{proof} 
\end{thm}

This theorem can be applied to all the general $\NIP$ cases with regard to additional structure being $\ntp_2$:

\begin{cor}
Algebraically closed valued fields with extra structure on the residue field and value group, with lift and section, are $\ntp_2$ if and only if the structure of the value group and the residue field are $\ntp_2$.
\end{cor}

\begin{cor}
Henselian valued fields in equicharacteristic $0$ with extra structure on the residue field and value group, with lift and section are $\ntp_2$ if and only if the structure of the value group and the residue field are $\ntp_2$.
\end{cor}

\begin{cor}
The theory of a multiplicative $\sigma$-henselian valued difference fields with lift, section and extra structure on the residue field and value group is $\ntp_2$ if and only if the theory of its residue field and value group are.
\end{cor}


\begin{thebibliography}{AVDD11}

\bibitem[AvdD11]{azgin2011elementary}Azgin, S. \& Van Den Dries, L. Elementary theory of valued fields with a valuation-preserving automorphism. {\em Journal Of The Institute Of Mathematics Of Jussieu}. \textbf{10}, 1-35 (2011)

\bibitem[Azg10]{azgin2010valued}Azgin, S. Valued fields with contractive automorphism and Kaplansky fields. {\em Journal Of Algebra}. \textbf{324}, 2757-2785 (2010)

\bibitem[B99]{MR1703196}Bélair, L. Types dans les corps valués munis d'applications coefficients. {\em Illinois J. Math.}. \textbf{43}, 410-425 (1999)

\bibitem[BMS07]{belair2007model}Bélair, L., Macintyre, A. \& Scanlon, T. Model theory of the Frobenius on the Witt vectors. {\em American Journal Of Mathematics}. \textbf{129}, 665-721 (2007)

\bibitem[Che76]{cherlin1976model}Cherlin, G. Model Theoretic Algebra 1. {\em The Journal Of Symbolic Logic}. \textbf{41}, 537-545 (1976)


\bibitem[CH14]{chernikov2014valued}Chernikov, A. \& Hils, M. Valued difference fields and $\ntp$. {\em Israel Journal Of Mathematics}. \textbf{204}, 299-327 (2014)


\bibitem[Che14]{chernikov2014theories}Chernikov, A. Theories without the tree property of the second kind. {\em Annals of Pure And Applied Logic}. \textbf{165}, 695-723 (2014)

\bibitem[vdD14]{vandenDries2014}Dries, L. Lectures on the Model Theory of Valued Fields. {\em Model Theory In Algebra, Analysis And Arithmetic: Cetraro, Italy 2012, Editors: H. Dugald Macpherson, Carlo Toffalori}. pp. 55-157 (2014)

\bibitem[EP05]{engler2005valued}Engler, A. \& Prestel, A. Valued fields. (Springer Science \& Business Media,2005)

\bibitem[JS20]{jahnke2016nip}Jahnke, F. \& Simon, P. NIP henselian valued fields. {\em Archive For Mathematical Logic}. \textbf{59}, 167-178 (2020)

\bibitem[Pal12]{pal2012multiplicative}Pal, K. Multiplicative valued difference fields. {\em The Journal of Symbolic Logic}. \textbf{77}, 545-579 (2012)

\bibitem[Pas89]{pas1989uniform}Pas, J. Uniform p-adic cell decomposition and local zeta functions.. {\em Journal für die reine und angewandte Mathematik}. \textbf{1989}, 137-172 (1989)


\end{thebibliography}
\end{document}